\documentclass[reqno,12pt]{amsart}
\def\bege{\begin{equation}} \def\ende{\end{equation}}
\def\begr{\begin{eqnarray}} \def\endr{\end{eqnarray}}

\usepackage{amsmath}
\usepackage{amssymb}
\usepackage{cite}
\usepackage{bbm}
\usepackage{amsmath}
\usepackage{txfonts}
\usepackage{stmaryrd}
\usepackage{amssymb}
\usepackage{amsfonts}
\usepackage{times}
\usepackage{mathrsfs}
\usepackage{amsthm}
\usepackage{enumerate}
\usepackage{framed}
\usepackage{lipsum}
\usepackage{color}

\newcommand{\DD}{{\mathbb D}}

\def\begr{\begin{eqnarray}} \def\endr{\end{eqnarray}}

\def\msk{\medskip}

\newtheorem{Lemma}{Lemma}
\newtheorem{Theorem}{Theorem}

\begin{document}
	\title[  ]{ 2-complex symmetric composition operators on $H^2$}
	\author{ Lian Hu,  Songxiao Li$^\ast$ and Rong Yang  }

	\address{Lian Hu\\ Institute of Fundamental and Frontier Sciences, University of Electronic Science and Technology of China,
		610054, Chengdu, Sichuan, P.R. China.}
	\email{202111210323@std.uestc.edu.cn  }
	\address{Songxiao Li\\ Institute of Fundamental and Frontier Sciences, University of Electronic Science and Technology of China,
		610054, Chengdu, Sichuan, P.R. China. }
	\email{lisongxiao@uestc.edu.cn}
	\address{Rong Yang\\ Institute of Fundamental and Frontier Sciences, University of Electronic Science and Technology of China,
		610054, Chengdu, Sichuan, P.R. China.}
	\email{201921210221@std.uestc.edu.cn}
	
	\begin{abstract}
		In this paper, we study 2-complex symmetric composition operators with the conjugation $J$ on the Hardy space $H^2$. More precisely, we obtain the necessary and sufficient condition for the composition operator $C_\phi$ to be 2-complex symmetric when the symbols $\phi$ is an automorphism of $\DD$. We also characterize the 2-complex symmetric  composition operator $C_\phi$  on the Hardy space $H^2$ when $\phi$ is a linear fractional self-map of $\DD$.
			\thanks{$\ast$ Corresponding author.}
		 \end{abstract}
\maketitle

\section{Introduction}
  Let $\DD$ and $\partial\DD$ be the open unit disk $\mathbb{C}$ and the unit circle in the complex plane, respectively. Let $H(\DD)$ be the set of all analytic functions on $\DD$.  The   Hardy  space $H^2(\DD)$ is the space of  all  $f\in H(\DD)$ such that $$
\|f\|^2=\sup_{0<r<1}{1\over 2\pi}\int_0^{2\pi}|f(re^{i\theta})|^2d\theta<\infty.$$
The space $H^2(\DD)$ is a reproducing kernel Hilbert space, that is, for each $w\in \DD$ and $f\in H^2(\DD)$, there is a unique function $K_w\in H^2(\DD)$ such that $$
\langle f,K_w\rangle=f(w),$$where $K_w={1\over 1-\bar{w}z}$ is said to be the reproducing kernel at $w$. For $g\in L^\infty(\partial\DD)$, the Toeplitz operator $T_g$ is defined as $T_gf=P(fg)$ for $f\in H^2(\DD)$, where $P$ is the orthogonal projection of $L^2$ onto $H^2(\DD)$. Recall that $$T_g^\ast K_w=\overline{f(w)}K_w$$ for each $w\in \DD$ and $f\in H^\infty$.

Let $\phi$ be an analytic self-map of $\DD$. Recall that  the composition operator $C_\phi$ is defined by
  $$C_\phi f(z)=f(\phi(z)), ~~~~~~~f\in H(\DD),~~~ z\in \DD.$$
   It is easy to see that  $C_\phi^\ast K_w=K_{\phi(w)}$ foe each $w\in \DD$.

Let $H$ and $\mathcal{B}(H)$ stand for a separable complex Hilbert space and the set of all continuous linear operators on $H$, respectively. An operator $C:H\to H$ is said to be a conjugation on $H$ if it is
	\begin{enumerate}
	\item[(a)] anti-linear or conjugate-linear: $C(\alpha +\beta y)=\bar{\alpha}C(x)+\bar{\beta}C(y)$, for all $\alpha,\beta\in \mathbb{C}$ and $x,y\in H$,
	\item[(b)] isometric: $\|Cx\|=\|x\|$, for all $x\in H$,
	\item[(c)] involutive: $C^2=I_d$, where $I_d$ is an identity operator.
	\end{enumerate}
It is easy to check that $(Jf)(z)=\overline{f(\bar{z})}$ is a conjugation on the space $H^2$.

An operator $T\in \mathcal{B}(H)$ is said to be complex symmetric if $ T=CT^\ast C$, where $C$ is a conjugation on $H$. We also say that $T$ is complex symmetric with $C$ or $C$-symmetric operator. The class of complex symmetric operators includes all normal operators, binormal operators, Hankel operators, compressed Toeplitz operators and the Volterra integration operator. The general study of complex symmetric operators was initiated  by Garcia, Putinar and Wogen in \cite{gp,gp2,gw,gw2}. See \cite{hw,lk,nst,gm,f,ns, tmh, y,gz}
for more results on complex symmetric operators.

 Let $m$ be a positive integer and $T\in \mathcal{B}(H)$. In \cite{ckl}, the authors  defined $m$-complex symmetric operators as follows:   $T $ is said to be an $m$-complex symmetric operator if there exists some conjugation $C$ such that
$$
\sum_{j=0}^{m}(-1)^{m-j}\left(\begin{array}{c}
m \\
j
\end{array}\right) T^{* j} C T^{m-j}C  =0.
$$
In this case, $T$ is said to be $m$-complex symmetric with $C$.  When   $m=2$,  we get
  $$CT^2C-2T^\ast CTC+{T^\ast}^2 =0, $$
which is equivalent to
$$ CT^2-2T^\ast CT+{T^\ast}^2C =0. $$ It is clear that $1$-complex symmetric operator is just the complex symmetric operator. By a simple calculation, we  see that all complex symmetric operators are 2-complex symmetric  operators. In \cite{ejjl}, the authors    studied $m$-complex symmetric  weighted shifts on $\mathbb{C}^{n}$.   We refer the reader to \cite{ckl, ckl2, ejjl, jr} for more details about $m$-complex symmetric operators.

In this paper, we study 2-complex symmetric composition operators with  $J$ on the Hardy  space $H^2$. When the symbols $\phi$ is an automorphism of $\DD$, we see that  the composition operator $C_\phi$ is 2-complex symmetric with  $J$ if and only if $C_\phi$ is  normal.   Furthermore, we also characterize the 2-complex symmetric  composition operators on $H^2$ when the induced map is a linear fractional self-map of $\DD$.\msk

\section{Main results}
  We begin this section with Cowen's formula for the adjoint of a linear fractional self-map. For a linear fractional self-map $\phi(z)={az+b\over cz+d}$ of $\DD$, where $a,b,c,d\in \mathbb{C}$ and $ad-bc\ne 0$, Cowen in \cite{c} obtained the following important formula $$C_\phi^\ast =T_gC_\sigma T_h^\ast,$$
   where
    $$
  \sigma(z)={\bar{a}z-\bar{c}  \over -\bar{b}z+\bar{d} },~~g(z)={1\over-\bar{b}z+\bar{d} },~~ h(z)=cz+d.$$

\begin{Lemma}\label{lem1}\cite{cm}
	Let $\phi$ be an analytic self-map of  $\DD$. Then $C_\phi$ is normal if and only if $\phi(z)=az$ with $|a|\le 1$.
\end{Lemma}

\begin{Lemma}\label{lem2}
	Let $\phi(z)={ az+b \over cz+d }$ be a linear fractional self-map of $\DD$. If $w\in\DD$ satisfies $\bar{a}w\ne\bar{c}$, then the following statements hold:
	\begin{enumerate}
		\item[(a)]  $JC_\phi^2K_w(z)=K_{\bar{w}}(\overline{\phi_2(\bar{z})})$,
		\item[(b)]
		$C_\phi^\ast JC_\phi K_{w}(z)=-\bar{c}{ g(w) \over \sigma(w) }K_{\phi(0)}(z)+\overline{h( {1  \over \overline{\sigma(w)} } )}g(w)K_{\phi( \overline{ \sigma(w) })}(z),$
		\item[(c)]${C_\phi^\ast}^2JK_w(z)=K_{\phi_2(\bar{w})}(z),$
		\item[(d)] $ JC_\phi JC_\phi^\ast J K_{w}(z)=K_{\phi(\bar{w})}(\overline{\phi(\bar{z})})$
	\end{enumerate}
where $\sigma(z)={ \bar{a}z-\bar{c} \over -\bar{b}z+\bar{d} }$, $g(z)={1\over -\bar{b}z+\bar{d}}$, and $h(z)=cz+d$.
\end{Lemma}

\begin{proof}
	Since $Jf(z)=\overline{f(\bar{z})}$ and $f(w)=\langle f,K_w\rangle$ for any $w\in \DD$ and $f\in H^2$, then $(a)$ and $(c)$ hold obviously. Now we only verify $(b)$ and $(d)$. For any $w\in\DD$ with $\bar{a}w\ne\bar{c}$, we obtain
	\begin{align*}
		C_\phi^\ast JC_\phi K_{w}(z)=&C_\phi^\ast K_{\bar{w}}(\overline{\phi(\bar{z})})\\
=&C_\phi^\ast{ \bar{c}z+\bar{d} \over (\bar{c}-\bar{a}w)z+(\bar{d}-\bar{b}w) }\\
	=&	C_\phi^\ast\left(  { \bar{c} \over \bar{c}-\bar{a}w}+\left( \bar{d}-{ \bar{c}(\bar{d}-\bar{b}w) \over \bar{c}-\bar{a}w } \right){1\over \bar{d}-\bar{b}w} {1\over  1-{ \bar{a}w-\bar{c} \over \bar{d}-\bar{b}w }z} \right)\\
	=& { \bar{c} \over \bar{c}-\bar{a}w}C_\phi^\ast K_0(z)+\left( \bar{d}-{ \bar{c}(\bar{d}-\bar{b}w) \over \bar{c}-\bar{a}w } \right){1\over \bar{d}-\bar{b}w}	C_\phi^\ast K_{\overline{\sigma(w)}}(z)\\
		=&-\bar{c}{ g(w) \over \sigma(w) }K_{\phi(0)}(z)+\overline{h( {1  \over \overline{\sigma(w)} } )}g(w)K_{\phi( \overline{ \sigma(w) })}(z).
	\end{align*}
Clearly,
 \begin{align*}
	JC_\phi JC_\phi^\ast J K_{w}(z)&=JC_\phi JC_\phi^\ast K_{\bar{w}}(z)=JC_\phi JK_{\phi(\bar{w})}(z)\\
	&=		JC_\phi K_{\overline{\phi(\bar{w})}}(z)=JK_{\overline{\phi(\bar{w})}}(\phi(z))\\&=K_{\phi(\bar{w})}(\overline{\phi(\bar{z})}).
\end{align*}
The proof is complete.
\end{proof}

\begin{Lemma}\label{lem3}
Let $\phi$ be an analytic self-map of $\DD$. If $C_\phi$ is normal, then  $C_\phi^\ast $ and $C_\phi$ is 2-complex symmetric with   $J$.
\end{Lemma}

\begin{proof}
	Assume that $C_\phi$ is normal. From Corollary 3.10 in  \cite{jkkl} we see that  $C_\phi$ is  complex symmetric with $J$.  Hence, $C_\phi$ is  2-complex symmetric with  $J$.	
	Write $\phi(z)=az$ with $|a|\le 1$ from Lemma \ref{lem1}. We notice that $T^*$ is   $2$-complex symmetric with $J$ if and only if
$$ {T^\ast}^2J-2JT JT^\ast J+ JT^2 =0. $$
  In view of Lemma \ref{lem2}, set $\sigma(z)=\bar{a}z$, $ g(z)=h(z)=1$, we have
	$$
	{C_\phi^\ast}^2JK_w(z)=K_{\phi_2(\bar{w})}(z)={1\over 1-\bar{a}^2wz},$$

	\begin{align*}
		JC_\phi JC_\phi^\ast J K_{w}(z)&=K_{\phi(\bar{w})}(\overline{\phi(\bar{z})})={1\over 1-\bar{a}^2wz}
	\end{align*}
and
$$
JC_\phi^2K_w(z)=K_{\bar{w}}(\overline{\phi_2(\bar{z})})={1\over 1-w\overline{\phi_2(\bar{z})}}={1\over 1-\bar{a}^2wz}.$$
	Thus, it is obvious that ${C^\ast_\phi}^2 J-2JC_\phi J C^\ast_\phi J + JC^2_\phi=0$, which means that $C_\phi^\ast$ is 2-complex symmetric with   $J$.
\end{proof}

\begin{Lemma}\label{lem4}
	Let $\phi$ be an analytic self-map of $\DD$. If $C_\phi$ is 2-complex symmetric with  $J$, then $$2\phi_2(0)-4\phi(0)+\phi(0)\phi_2(0)=0.$$
\end{Lemma}

\begin{proof}
	Since $f(w)=\langle f,K_w\rangle$ for any $w\in \DD$ and $f\in H^2$, we have
  $$\langle JC_\phi^2K_0,K_{1\over 2} \rangle=\langle JK_0,K_{1\over 2}\rangle=1,$$ $$\langle
	C_\phi^\ast JC_\phi K_0,K_{1\over 2}\rangle=\langle C_\phi^\ast K_0, K_{1\over 2}\rangle=\langle K_{\phi(0)},K_{1\over 2}\rangle={2\over 2-\overline{\phi(0)}}$$and$$
	\langle {C_\phi^\ast}^2JK_0, K_{1\over 2}\rangle= \langle {C_\phi^\ast}^2K_0,  K_{1\over 2}\rangle=\langle K_{\phi_2(0)},K_{1\over 2} \rangle={2\over 2-\overline{\phi_2(0)}}.$$
By the assumption that $C_\phi$ is 2-complex symmetric with   $J$, we get  that  $$1-{4\over 2-\overline{\phi(0)}}+{2\over 2-\overline{\phi_2(0)}}=0.$$  By a simple calculation, we get the desired result.
\end{proof}

\begin{Theorem}
	Let $\phi$ be an automorphism of $\DD$. Then $C_\phi$ is 2-complex symmetric with   $J$ if and only if $C_\phi$ is normal.
\end{Theorem}

\begin{proof} Assume that $C_\phi$ is normal. From Lemma 3 we see that   $C_\phi$ is  2-complex symmetric with   $J$.

	Now suppose that $\phi(z)=\lambda{ a-z \over 1-\bar{a}z }$ with $|\lambda|=1$, $a\in \DD$ and $C_\phi$ is 2-complex symmetric with   $J$.
Let $\sigma(z)={ -\bar{\lambda}z+a\over -\overline{\lambda a}z+1}$, $g(z)={1\over -\overline{\lambda a}z+1}$ and $h(z)=-\bar{a}z+1$. Noting that $$\phi_2(z)={\lambda(\lambda z-\lambda a)-\lambda a(\bar{a}z-1)  \over \bar{a}(\lambda z-\lambda a)- (\bar{a}z-1)},$$
  then for $\bar{\lambda}w\ne a$, Lemma \ref{lem2} gives that
	\begin{align}\label{equ1}
		JC_\phi^2K_w(z)=K_{\bar{w}}(\overline{\phi_2(\bar{z})})={ a(\bar{\lambda}z-\overline{\lambda a})-(az-1) \over a(\bar{\lambda}z-\overline{\lambda a})-(az-1)-w[\bar{\lambda}(\bar{\lambda}z-\overline{\lambda a})-\overline{\lambda a}(az-1)] },
		\end{align}
  \begin{equation}\label{equ2}
  	\begin{aligned}
  	C_\phi^\ast JC_\phi K_{w}(z)=&a{ g(w) \over \sigma(w) }K_{\phi(0)}(z)+\overline{h( {1  \over \overline{\sigma(w)} } )}g(w)K_{\phi( \overline{ \sigma(w) })}(z)\\
  	=& {a\over (-\bar{\lambda} w+a)(1-\overline{\lambda a}z)}+{\bar{\lambda }w(|a|^2-1) \over (-\bar{\lambda }w+a)(-\overline{\lambda a}w+1) }\\
  	&\cdot {-\overline{\lambda a }w+1+a\bar{\lambda }w-{a}^2  \over -\overline{\lambda a }w+1+a\bar{\lambda }w-{a}^2 -(\bar{\lambda}^2{w}-\bar{\lambda}^2\bar{a}^2{w}-\bar{\lambda}a+\overline{\lambda a})z}
  \end{aligned}
  \end{equation}
and
\begin{align}\label{equ3}
 {C_\phi^\ast}^2JK_w(z)=K_{\phi_2(\bar{w})}(z)={ a(\bar{\lambda}w-\overline{\lambda a})-(aw-1) \over a(\bar{\lambda}w-\overline{\lambda a})-(aw-1)-[\bar{\lambda}(\bar{\lambda}w-\overline{\lambda a})-\overline{\lambda a}(aw-1)]z }
 \end{align}
for any $w,z\in \DD$.
 Taking $w=0$ in (\ref{equ1}), (\ref{equ2}) and (\ref{equ3}), we have
 $$
 JC_\phi^2K_0(z)=1,\,\,\,\,\,\, C_\phi^\ast JC_\phi K_{0}(z)={ 1\over 1-\overline{\lambda a}z},$$and$$
  {C_\phi^\ast}^2JK_0(z)={-\bar{\lambda}|a|^2+1  \over  -\bar{\lambda}|a|^2+1+(\bar{\lambda}^2\bar{a}-\overline{\lambda a})z }$$for any $z\in \DD$.
 Since  $C_\phi$ is 2-complex symmetric with   $J$, we get
  \begin{align}\label{equ5}
  { 1+\overline{\lambda a}z \over 1-\overline{\lambda a}z }={1-\bar{\lambda}|a|^2 \over  -\bar{\lambda}|a|^2+1+(\bar{\lambda}^2\bar{a}-\overline{\lambda a})z}.
  \end{align}
Calculating and noting that the coefficients of $z^2$ must be $0$, we obtain that $\lambda^2 a^2(\lambda-1)=0$. Since $\lambda$ and $a$ are nonzero, then $\lambda=1$. Hence, (\ref{equ5}) becomes $$
{ 1+\bar{a}z\over 1-\bar{a}z }=1,$$
 which implies that $\bar{a}z=0$ for all $z\in \DD$. Thus $a=0$. Hence
  $\phi(z)=-\lambda z$ with $|\lambda|=1$. Lemma \ref{lem1}   gives that $C_\phi$ is normal.  The proof is complete.
\end{proof}

\section{Linear fractional self-maps}
In this section, we first consider 2-complex symmetric composition operators with  $J$ which are induced by linear fractional self-maps with $\phi(0)=0$.

\begin{Theorem}\label{the2}
	Let $\phi(z)={az+b\over cz+d }$ be a linear fractional of $\DD$ and $\phi(0)=0$. Then $C_\phi$ ($C_\phi^\ast $) is 2-complex symmetric with  $J$ if and only if $C_\phi$ is normal.
\end{Theorem}

\begin{proof}
	Assume first that $C_\phi$  is normal. Lemma \ref{lem3} gives that $C_\phi$ ($C_\phi^\ast $) is 2-complex symmetric with  $J$.
	
	Conversely, suppose that $C_\phi$ is 2-complex symmetric with   $J$. Since $a\ne 0$ and $\phi(0)=0 $, set $\phi(z)={z\over sz+t}$, where $s={c\over a}$ and $t={d\over a}$ satisfy $ |t|\ge |s|+1$. If $ s=0$, then $\phi(z)={1\over t}z$. Therefore, Lemma \ref{lem1} gives that $C_\phi$ is normal. Now, we suppose that $s\ne 0$. Let $\sigma(z)={z-\bar{s}\over \bar{t}}$, $g(z)={1\over \bar{t}}$ and $h(z)=sz+t$. Note that $$\phi_2(z)={z\over (s+st)z+t^2}$$ and $$
	\phi(\overline{\sigma(z)})={ \bar{z}-s \over s\bar{z}-s^2+t^2 },\,\,\ h\left( { 1\over \overline{\sigma(z)}}\right)={ st\over \bar{z}-s}+t.$$
	For any $w,z\in \DD$ with $w\ne \bar{s}$, employing Lemma \ref{lem2}, we obtain that
	\begin{align}\label{equ7}
	 JC_\phi^2K_w(z)=K_{\bar{w}}(\overline{\phi_2(\bar{z})})={1\over 1-{zw  \over \bar{s}z+\overline{st}z+\bar{t}^2 }}={ (\bar{s}+\overline{st})z+\bar{t}^2 \over (\bar{s}+\overline{st})z+\bar{t}^2-wz },
	 \end{align}
	 \begin{equation}\label{equ8}
		\begin{aligned}
		C_\phi^\ast JC_\phi K_{w}(z)=&-\bar{s}{ g(w) \over \sigma(w) }K_{\phi(0)}(z)+\overline{h( {1  \over \overline{\sigma(w)} } )}g(w)K_{\phi( \overline{ \sigma(w) })}(z)\\
		=& {-\bar{s}\over w-\bar{s}}+\left( {\overline{st}\over {w}-\bar{s}}+\bar{t}  \right)\cdot{1\over \bar{t}}{ 1\over 1-{w-\bar{s}  \over \bar{s}w-\bar{s}^2+\bar{t}^2 }z}\\
		=&{\bar{s}\over  \bar{s}-w}+{ w\bar{t} \over (w-\bar{s} )\bar{t}}{ \bar{s}w-\bar{s}^2+\bar{t}^2 \over \bar{s}w-\bar{s}^2+\bar{t}^2- (w-\bar{s})z},
	\end{aligned}
\end{equation}
and 	\begin{align}\label{equ9}
 {C_\phi^\ast}^2JK_w(z)=K_{\phi_2(\bar{w})}(z)={1\over 1-{zw  \over \bar{s}w+\overline{st}w+\bar{t}^2 }}={ (\bar{s}+\overline{st})w+\bar{t}^2 \over (\bar{s}+\overline{st})w+\bar{t}^2-wz }.
  \end{align}
Since $C_\phi$ is 2-complex symmetric with   $J$, we obtain from (\ref{equ7}),  (\ref{equ8}) and (\ref{equ9}) that
  $$
  { (\bar{s}+\overline{st})z+\bar{t}^2 \over (\bar{s}+\overline{st}-w)z+\bar{t}^2 }+{ (\bar{s}+\overline{st})w+\bar{t}^2 \over (\bar{s}+\overline{st})w+\bar{t}^2-wz }
  $$
  $$={2\bar{s}\over  \bar{s}-w}+{ 2w\bar{t} \over (w-\bar{s} )\bar{t}}{ \bar{s}w-\bar{s}^2+\bar{t}^2 \over \bar{s}w-\bar{s}^2+\bar{t}^2- (w-\bar{s})z}$$
  for any $w,z\in\DD$ with $w\ne \bar{s}$, which gives that
 \begin{equation}\label{equ11}
 	\begin{aligned}
   &{ [(\bar{s}+\overline{st})z+\bar{t}^2 ] [(\bar{s}+\overline{st})w+\bar{t}^2-wz]+[(\bar{s}+\overline{st})w+\bar{t}^2] [ (\bar{s}+\overline{st}-w)z+\bar{t}^2] \over [(\bar{s}+\overline{st}-w)z+\bar{t}^2] [(\bar{s}+\overline{st})w+\bar{t}^2-wz]  }\\
   =&{ 2\overline{st}(w-\bar{s} )[\bar{s}w-\bar{s}^2+\bar{t}^2- (w-\bar{s})z] +2w\bar{t}(\bar{s}-w)(\bar{s}w-\bar{s}^2+\bar{t}^2)\over -(w-\bar{s})^2\bar{t} [\bar{s}w-\bar{s}^2+\bar{t}^2- (w-\bar{s})z]}
   \end{aligned}
\end{equation}
for any $w,z\in\DD$ with $w\ne \bar{s}$. Taking $w=0$ in (\ref{equ11}), then
$$
	{\bar{t}^2[(\bar{s}-\overline{st})z+\bar{t}^2] +\bar{t}^2[(\bar{s}+\overline{st})z+\bar{t}^2] \over  \bar{t}^2[(\bar{s}+\overline{st})z+\bar{t}^2]}={ 2\bar{s}^2\bar{t}(-\bar{s}^2+\bar{t}^2+\bar{s}z) \over \bar{s}^2\bar{t}(-\bar{s}^2+\bar{t}^2+\bar{s}z) }=2.
$$
 Therefore, $ (\bar{s}+\overline{st})z+\bar{t}^2 =(\bar{s}-\overline{st})z+\bar{t}^2$ for all $z\in \DD$. $s\ne0$ gives that $1+t=1-t$. Hence $t=0$,  a contradiction. So the hypothesis is not true, that is $s=0$.

Now, assume that $C_\phi^\ast$ is 2-complex symmetric with  $J$. We also assume that $s\ne 0$.   After a calculation, we have
\begin{equation}\label{equ31}
	\begin{aligned}
	JC_\phi JC_\phi^\ast J K_{w}(z)&=K_{\phi(\bar{w})}(\overline{\phi(\bar{z})})={1\over 1-{  w\over \bar{s}w+\bar{t}}\cdot {z  \over \bar{s}z+\bar{t} }}\\
& ={ (\bar{s}w+\bar{t}) (\bar{s}z+\bar{t})\over (\bar{s}w+\bar{t}) (\bar{s}z+\bar{t})-wz}
	\end{aligned}
\end{equation}
for any $w,z\in \DD$.
Since $C_\phi^\ast$ is 2-complex symmetric with  $J$, we obtain from (\ref{equ7}), (\ref{equ9}) and (\ref{equ31}) that
  \begin{equation}\label{equ10}
 	\begin{aligned}
 		&{ [(\bar{s}+\overline{st})z+\bar{t}^2 ] [(\bar{s}+\overline{st})w+\bar{t}^2-wz]+[(\bar{s}+\overline{st})w+\bar{t}^2] [ (\bar{s}+\overline{st}-w)z+\bar{t}^2] \over [(\bar{s}+\overline{st}-w)z+\bar{t}^2] [(\bar{s}+\overline{st})w+\bar{t}^2-wz]  }\\
 		=&{2 (\bar{s}w+\bar{t}) (\bar{s}z+\bar{t})\over (\bar{s}w+\bar{t}) (\bar{s}z+\bar{t})-wz}
 	\end{aligned}
 \end{equation}
 for any $w,z\in\DD$ with $w\ne \bar{s}$. Taking $w=0$ in (\ref{equ10}), we also have that
 $$
 {\bar{t}^2[(\bar{s}-\overline{st})z+\bar{t}^2] +\bar{t}^2[(\bar{s}+\overline{st})z+\bar{t}^2] \over  \bar{t}^2[(\bar{s}+\overline{st})z+\bar{t}^2]}=2.
 $$
 The other arguments are similar to the case of  $C_\phi$.  Then we obtain the desired result. The proof is complete.
\end{proof}

 We now consider 2-complex symmetric composition operators with   $J$ which are induced by linear fractional self-maps with $c=0$.

\begin{Lemma}\label{lem5}\cite{jkk}
	Let $\phi(z)={az+b\over cz+d }$ be a linear fractional of $\DD$ such that $c=0$. Then $C_\phi=C_\xi^\ast T_\tau^\ast$, where $\xi(z)={\bar{a}z  \over -\bar{b}z+\bar{d} }$ and $\tau(z)={ \bar{d} \over -\bar{b}z+\bar{d}}$.
\end{Lemma}

\begin{Theorem}\label{the3}
	Let $\phi(z)={az+b\over cz+d }$ be a linear fractional of $\DD$ such that $c=0$. Then $C_\phi$ is 2-complex symmetric with   $J$ if and only if $C_\phi$ is normal.
\end{Theorem}

\begin{proof}
	Assume first that $C_\phi$ is normal. Lemma \ref{lem3} gives that $C_\phi$ is 2-complex symmetric with   $J$.
	
	Conversely, suppose that $C_\phi$ is 2-complex symmetric with   $J$. Since $c=0$, then $\phi(z)=mz+n$, where $m={a\over d}$ and $n={b\over d}$ satisfy $|m|+|n|\le1$. Let $\xi(z)={ \bar{m}z \over -\bar{n}z+1 }$ and $\tau(z)={1\over -\bar{n}z+1}$. Note that $$\phi_2(z)=m^2z+mn+n$$ and$$
	\xi_2(z)={\bar{m}^2z \over -\overline{mn}z-\bar{n}z+1 },\,\,\,\phi(\overline{\xi(z)})={ m^2\bar{z}-n^2\bar{z}+n \over -n\bar{z}+1 },\,\,\,\tau(\xi(z))={ -\bar{n}z+1 \over -\overline{mn}z-\bar{n}z+1 }. $$ Lemmas \ref{lem2} and \ref{lem5} give that
	\begin{equation}\label{equ12}
		\begin{aligned}
		JC_\phi^2K_w(z)&=JC_\xi^\ast T_\tau^\ast C_\xi^\ast T_\tau^\ast K_w(z)=J\overline{\tau(w)}C_\xi^\ast T_\tau^\ast K_{\xi(w)}(z)\\
		&=J\overline{\tau(w)\tau(\xi(w))}K_{\xi_2(w)}(z)=\tau(w)\tau(\xi(w))K_{\overline{\xi_2(w)}}(z)\\
		&={1\over -\bar{n}w+1}{-\bar{n}w+1  \over -\overline{mn}w-\bar{n}w+1 }\cdot{ 1\over 1-{ \bar{m}^2w \over -\overline{mn}w-\bar{n}w+1 }z}\\
		&={ 1 \over -\overline{mn}w-\bar{n}w+1-\bar{m}^2wz },
		\end{aligned}
	\end{equation}
	\begin{equation}\label{equ13}
	\begin{aligned}
	C_\phi^\ast JC_\phi K_{w}(z)&=C_\phi^\ast JC_\xi^\ast T_\tau^\ast K_{w}(z)=	C_\phi^\ast J\overline{\tau(w)}K_{\xi(w)}(z)=	C_\phi^\ast\tau(w)\overline{K_{\xi(w)}(\bar{z})}\\&=C_\phi^\ast\tau(w)K_{\overline{\xi(w)}}(z)	=\tau(w)K_{\phi(\overline{\xi(w)})}(z)\\
	&={1\over -\bar{n}w+1}\cdot {1\over 1-{ \bar{m}^2w-\bar{n}^2w+\bar{n} \over -\bar{n}w+1 }z}\\
	&={  1\over -\bar{n}w+1-(\bar{m}^2w-\bar{n}^2w+\bar{n})z}
	\end{aligned}
\end{equation}
and
\begin{align}\label{equ14}
	{C_\phi^\ast}^2JK_w(z)=K_{\phi_2(\bar{w})}(z)={1\over 1-(\bar{m}^2w+\overline{mn}+\bar{n})z}
\end{align}
for any $w,z\in\DD$. Since $C_\phi$ is 2-complex symmetric with  $J$, then we obtain from (\ref{equ12}), (\ref{equ13}) and (\ref{equ14}) that
\begin{equation}\label{equ15}
	\begin{aligned}
& { 1 \over -\overline{mn}w-\bar{n}w+1-\bar{m}^2wz }+{1\over 1-(\bar{m}^2w+\overline{mn}+\bar{n})z}\\
&= {  2\over -\bar{n}w+1-(\bar{m}^2w-\bar{n}^2w+\bar{n})z}
\end{aligned}
\end{equation}
for any $w,z\in\DD$. Taking $w=0$ in (\ref{equ15}), we have that $$ 1+{1\over 1-(\overline{mn}+\bar{n})z}={2\over 1-\bar{n}z}$$
 for any $z\in\DD$. Then $${2-(\overline{mn}+\bar{n})z\over 1-(\overline{mn}+\bar{n})z }={2\over 1-\bar{n}z}.$$   Noting that the coefficient of $z^2$ must be $0$, then $n^2(m+1)=0$, which means that $n=0$ or $m=-1$. Similarly, noting that the coefficient of $z$ must be also $0$, then $n=mn$, which means that $n=0$ or $m=1$. Therefore,  $n=0$. This implies that $\phi(z)=mz$ with $|m|\le1$. Lemma \ref{lem1} gives that $C_\phi$ is normal.
\end{proof}

\begin{Theorem}
	Let $\phi(z)={az+b\over cz+d }$ be a linear fractional of $\DD$ such that $c=0$. Then $C_\phi^\ast$ is 2-complex symmetric with   $J$ if and only if $C_\phi$ is normal.
\end{Theorem}

\begin{proof}
	Since $c=0$, then set $\phi(z)=mz+n$ where $m={a\over d}$ and $n={b\over d}$ satisfy $|m|+|n|\le 1$. Let $\xi(z)={\bar mz  \over -\bar{n}z+1 }$ and $ \tau(z)={1\over -\bar{n}z+1}$.
	Noting that $$\phi_2(z)=m^2z+mn+n,$$
  and $$
	\tau(\overline{\phi(\bar{z})})={1\over -\bar{n}(\bar{m}z+\bar{n})+1},\,\,\,\,\xi(\overline{\phi(\bar{z})})={ \bar{m}(\bar{m}z+\bar{n}) \over -\bar{n}(\bar{m}z+\bar{n})+1 }.$$
	Using Lemma \ref{lem5} and Theorem \ref{the3}, we obtain that
		\begin{equation}\label{equ27}
		\begin{aligned}
			JC_\phi^2K_w(z)=JC_\xi^\ast T_\tau^\ast C_\xi^\ast T_\tau^\ast K_w(z)
			&={ 1 \over -\overline{mn}w-\bar{n}w+1-\bar{m}^2wz  },
		\end{aligned}
	\end{equation}
	\begin{equation}\label{equ28}
		\begin{aligned}
			JC_\phi JC_\phi^\ast J K_{w}(z)&=JC_\xi^\ast T_\tau^\ast K_{\overline{\phi(\bar{w})}}(z)\\
			&=J\overline{\tau( \overline{\phi(\bar{w})})}K_{\xi(\overline{\phi(\bar{w})})}(z)=\tau( \overline{\phi(\bar{w})})K_{\overline{ \xi(\overline{\phi(\bar{w})}) }}(z)\\
			&={1\over -\bar{n}(\bar{m}w+\bar{n})+1}\cdot{1\over  1-{ \bar{m}(\bar{m}w+\bar{n})z \over -\bar{n}(\bar{m}w+\bar{n})+1}}
		\\&={1\over 1-(\bar{m}w+\bar{n})(\bar{m}z+\bar{n})}
		\end{aligned}
	\end{equation}
and
	\begin{align}\label{equ29}
		{C_\phi^\ast}^2JK_w(z)=K_{\phi_2(\bar{w})}(z)={1\over 1-(\bar{m}^2w+\overline{mn}+\bar{n})z}
	\end{align}
	for any $w,z\in \DD$. Since $C_\phi^\ast$ is 2-complex symmetric with  $J$, then we obtain from (\ref{equ27}), (\ref{equ28}) and (\ref{equ29}) that
	\begin{align}\label{equ30}
	{ 1 \over -\overline{mn}w-\bar{n}w+1-\bar{m}^2wz  }+{1\over 1-(\bar{m}^2w+\overline{mn}+\bar{n})z}= {2\over 1-(\bar{m}w+\bar{n})(\bar{m}z+\bar{n})}
\end{align}
	for any $w,z\in \DD$. Taking $w=0$, we get
  $${2-(\overline{mn}+\bar{n})z\over 1-(\overline{mn}+\bar{n})z }={2\over 1-\bar{n}(\bar{m}z+\bar{n})}.$$
  Noting that the coefficient of constant term must be $0$, then $\bar{n}^2=0$, which means that $n=0$. Thus, $\phi(z)=mz$ with $|m|\le 1$. Lemma \ref{lem1} deduces that $C_\phi$ is normal.
	
Conversely, assume that $C_\phi$ is normal. Lemma \ref{lem3} gives that $C_\phi^\ast$ is 2-complex symmetric with   $J$.
\end{proof}

Furthermore, we prove that there is no 2-complex symmetric composition operators with   $J$ which are induced by linear fractional self-maps with  $a=0$ and $\phi(0)\ne 0$.

\begin{Lemma}\label{lem6}
	Let $\phi:\DD\to\DD$ be a constant function.  Then $C_\phi$ $ (C_\phi^\ast)$ is 2-complex symmetric with  $J$ if and only if $\phi(z)\equiv 0$.
\end{Lemma}

\begin{proof} The sufficiency is obvious.

	Now we assume   that $\phi(z)\equiv c$ for some $c\in \DD$ and $C_\phi$   is 2-complex symmetric with   $J$. Then we obtain that $$
	JC_\phi^2K_w(z)=JK_w(\phi_2(z))=J{1\over 1-\bar{w}c}={1\over 1-w\bar{c}},$$
	$$
	C_\phi^\ast JC_\phi K_{w}(z)=C_\phi^\ast J K_{w}(\phi(z))={1\over 1-w\bar{c}}	C_\phi^\ast K_0(z)={1\over (1-w\bar{c})(1-\bar{c}z)},	$$ $$
	{C_\phi^\ast}^2JK_w(z)=K_{\phi_2(\bar{w})}(z)={1\over 1-\bar{c}z} $$and$$
	JC_\phi JC_\phi^\ast J K_{w}(z)=K_{\phi(\bar{w})}(\overline{\phi(\bar{z})})={1\over 1-\bar{c}^2}  $$
	for any $w,z\in\DD$. Since $C_\phi$ is 2-complex symmetric with   $J$, we get
	\begin{align}\label{equ32}
{1\over 1-w\bar{c}}-{2\over (1-w\bar{c})(1-\bar{c}z)}+{1\over 1-\bar{c}z} =0
	\end{align}
for any $w,z\in\DD$.   By a simple calculation, we see that $c=0$.

	Assume that $C_\phi^\ast$ is 2-complex symmetric with   $J$. Similarly, we have
	\begin{align}\label{equ16}
		{1\over 1-w\bar{c}}-{2\over 1-\bar{c}^2}+{1\over 1-\bar{c}z} =0
	\end{align}
	for any $w,z\in\DD$.  Therefore, $c=0$. The proof is complete.
	\end{proof}

\begin{Theorem}
	Let $\phi(z)={az+b\over cz+d }$ be a linear fractional of $\DD$ such that $a=0$ and $\phi(0)\ne 0$. Then $C_\phi$ is not 2-complex symmetric with   $J$.
\end{Theorem}

\begin{proof}
	Since $a=0$ and $\phi(0)\ne 0$, we set $\phi(z)={1\over mz+n}$,  where $m={c\over b}$ and $n={d\over b}$. When $m=0$, $\phi(z)={1\over n}\ne0$, then Lemma \ref{lem6} gives that $C_\phi$ is not 2-complex symmetric with  $J$. Now, we assume that $m\ne 0$, $n\ne 0$
	and $C_\phi$ is 2-complex symmetric with  $J$. Let $\sigma(z)={ -\bar{m} \over -z+\bar{n} }$, $g(z)={1\over -z+\bar{n}}$ and $h(z)=mz+n$. Note that $$
	\phi_2(z)={ mz+n\over m+mnz+n^2},\,\, h\left({1\over \overline{\sigma(z)}}\right)=\bar{z},\,\, \phi(\overline{\sigma(z)})={-\bar{z}+n\over -m^2-n\bar{z}+n^2}.$$
	Lemma \ref{lem2} gives that
	\begin{align}\label{equ17}
		JC_\phi^2K_w(z)=K_{\bar{w}}(\overline{\phi_2(\bar{z})})={1\over 1-w{\bar{m}z+\bar{n}  \over \bar{m}+\overline{mn}z+\bar{n}^2 }}={\bar{m}+\overline{mn}z+\bar{n}^2 \over \bar{m}+\overline{mn}z+\bar{n}^2 -(\bar{m}z+\bar{n})w },
	\end{align}
	\begin{equation}\label{equ18}
		\begin{aligned}
			C_\phi^\ast JC_\phi K_{w}(z)=&-\bar{m}{ g(w) \over \sigma(w) }K_{\phi(0)}(z)+\overline{h( {1  \over \overline{\sigma(w)} } )}g(w)K_{\phi( \overline{ \sigma(w) })}(z)\\
			=& K_{\phi(0)}(z)+{ w \over -w+ \bar{n}}K_{\phi( \overline{ \sigma(w) })}(z)\\
		=&{1\over 1-{z\over\bar{n} }}+	{w \over -w+ \bar{n}}\cdot {1\over 1-{ -w+\bar{n} \over -\bar{m}^2-\bar{n}w+\bar{n}^2 }z}\\
		=&{\bar{n}\over \bar{n}-z}	+{ w \over -w+ \bar{n}}\cdot{  -\bar{m}^2-\bar{n}w+\bar{n}^2\over -\bar{m}^2-\bar{n}w+\bar{n}^2-(-w+ \bar{n})z }
		\end{aligned}
	\end{equation}
	and 	\begin{align}\label{equ19}
		{C_\phi^\ast}^2JK_w(z)=K_{\phi_2(\bar{w})}(z)={1\over 1-z{\bar{m}w+n  \over \bar{m}+\overline{mn}w+\bar{n}^2 }}={\bar{m}+\overline{mn}w+\bar{n}^2 \over \bar{m}+\overline{mn}w+\bar{n}^2 -(\bar{m}w+\bar{n})z}
	\end{align}
	for any $w, z\in \DD$. Since $C_\phi$ is 2-complex symmetric with  $J$,   we obtain from (\ref{equ17}), (\ref{equ18}) and (\ref{equ19}) that
	\begin{equation}\label{equ20}
		\begin{aligned}
	&	{\bar{m}+\overline{mn}z+\bar{n}^2 \over \bar{m}+\overline{mn}z+\bar{n}^2 -(\bar{m}z+\bar{n})w }+{\bar{m}+\overline{mn}w+\bar{n}^2 \over \bar{m}+\overline{mn}w+\bar{n}^2 -(\bar{m}w+n)z}	\\
	=&2\left({\bar{n}\over \bar{n}-z}	+{w \over -w+ \bar{n}}\cdot{  -\bar{m}^2-\bar{n}w+\bar{n}^2\over -\bar{m}^2-\bar{n}w+\bar{n}^2-(-w+ \bar{n})z }\right)
		\end{aligned}
	\end{equation}
	for any $w,z\in\DD$. Taking $w=0$ in (\ref{equ20}), we have that
	\begin{equation}\label{equ21}
	\begin{aligned}	
		{ (-\bar{n}z+\bar{m}+\bar{n}^2)+(\bar{m}+\bar{n}^2) \over -\bar{n}z+\bar{m}+\bar{n}^2 }={2\bar{n}\  \over\bar{n}-kz }
			\end{aligned}
	\end{equation}
	for any $z\in\DD$. Thus the coefficient of $z^2$ must be $0$. This implies that $n=0$, which is a contradiction. The proof is complete.
\end{proof}

In the remainder of this paper, we consider 2-complex symmetric composition operators with   $J$ which are induced by linear fractional self-maps with $a\ne 0$, $c\ne 0$ and $\phi(0)\ne 0$.

\begin{Theorem}
	 Let $\phi(z)={az+b\over cz+d }$ be a linear fractional of $\DD$ such that $a\ne 0$, $c\ne 0$ and $\phi(0)\ne 0$. Then $C_\phi$ is not 2-complex symmetric with $J$.
\end{Theorem}

\begin{proof} 	We prove it by contradiction.  Assume that $C_\phi$ is  2-complex symmetric with   $J$.
	Since $a\ne 0$, $c\ne 0$ and $\phi(0)\ne 0$, set $\phi(z)={mz+n\over sz+1 }$, where $m={a\over d}$, $n={b\over d}$ and $s={c\over d}$. Let $\sigma(z)={\bar{m}z-\bar{s}\over -\bar{n}z+1}$, $g(z)={1\over-\bar{n}z+1 }$ and $ h(z)=sz+1$. Note that $$
	\phi_2(z)={(m^2+ns)z+mn+n  \over (ms+s)z+ns+1 },\,~~~~~~~\,h\left({1\over \overline{\sigma(z)}}\right)={ (m-ns)\bar{z}\over m\bar{z}-s }$$and$$
	\phi(\overline{\sigma(z)})={ (m^2-n^2)\bar{z}+n-ms \over (ms-n)\bar{z}+1-s^2 }.$$
	Lemma \ref{lem2} gives that
\begin{equation}\label{equ22}
	\begin{aligned}
		JC_\phi^2K_w(z)&=K_{\bar{w}}(\overline{\phi_2(\bar{z})})={1\over 1-w{(\bar{m}^2+\overline{ns})z+\overline{mn}+\bar{n}  \over (\overline{ms}+\bar{s})z+\overline{ns}+1 }}\\&={(\overline{ms}+\bar{s})z+\overline{ns}+1  \over (\overline{ms}+\bar{s})z+\overline{ns}+1-[(\bar{m}^2+\overline{ns})z+\overline{mn}+\bar{n} ]w},
	\end{aligned}
\end{equation}
	\begin{equation}\label{equ23}
		\begin{aligned}
			C_\phi^\ast JC_\phi K_{w}(z)=&-\bar{s}{ g(w) \over \sigma(w) }K_{\phi(0)}(z)+\overline{h( {1  \over \overline{\sigma(w)} } )}g(w)K_{\phi( \overline{ \sigma(w) })}(z)\\
			=& {-\bar{s}\over \bar{m}w-\bar{s}}{1\over 1-\bar{n}z}+{ (\bar{m}-\overline{ns})w \over \bar{m}w-\bar{s} }{1\over -\bar{n}w+1}\cdot{1\over 1-{(\bar{m}^2-\bar{n}^2)w-\overline{ms}+\bar{n}  \over (\overline{ms}-\bar{n})w+1-\bar{s}^2 }z}\\
		=& { -\bar{s} \over (\bar{m}w-\bar{s})(1-\bar{n}z) }+{ (\bar{m}-\overline{ns})w \over (\bar{m}w-\bar{s})(-\bar{n}w+1) }\\
		&\,\,\,\,\,\,\,\,\,\,\,\,\,\,\,\,\,\,\,\,\,\,\,\,\,\,\,\,\,\,\,\,\,\,\,\,\,\cdot{ (\overline{ms}-\bar{n})w+1-\bar{s}^2 \over (\overline{ms}-\bar{n})w+1-\bar{s}^2-[(\bar{m}^2-\bar{n}^2)w-\overline{ms}+\bar{n}]z}	
		\end{aligned}
	\end{equation}
	and 	
	\begin{equation}\label{equ24}
		\begin{aligned}
		{C_\phi^\ast}^2JK_w(z)&=K_{\phi_2(\bar{w})}(z)={1\over 1-z{(\bar{m}^2+\overline{ns})w+\overline{mn}+\bar{n}  \over (\overline{ms}+\bar{s})w+\overline{ns}+1 }}\\&={(\overline{ms}+\bar{s})w+\overline{ns}+1  \over (\overline{ms}+\bar{s})w+\overline{ns}+1-[(\bar{m}^2+\overline{ns})w+\overline{mn}+\bar{n} ]z}
	\end{aligned}
\end{equation}
	for any $w,z\in \DD$. Taking $w=0$, then we obtain from (\ref{equ22}), (\ref{equ23}) and (\ref{equ24}) that$$
	JC_\phi^2K_0(z)=1,\,\,\,\,  C_\phi^\ast JC_\phi K_{0}(z)={1\over 1-\bar{n}z}	$$ and $$
	{C_\phi^\ast}^2JK_0(z)={\overline{ns}+1  \over \overline{ns}+1-(\overline{mn}+\bar{n})z }$$for any $z\in \DD$.
Since $C_\phi$ is 2-complex symmetric with  $J$, we obtain that
	 \begin{align}
	 	1+{\overline{ns}+1  \over \overline{ns}+1-(\overline{mn}+\bar{n})z }={2\over 1-\bar{n}z}
	 	\end{align}
 	for any $z\in \DD$, which implies that
 	\begin{align}\label{equ25}
 	 { 2\overline{ns}+2-(\overline{mn}+\bar{n})z \over \overline{ns}+1 -(\overline{mn}+\bar{n})z}={ 2\over 1-\bar{n}z}\end{align}
 	 for any $z\in \DD$. Noting that the coefficients of $z^2$ must be $0$, then we have that $\bar{n}^2(m+1)=0$. Since $\phi(0)\ne 0$, then  $m=-1$.

  Noting that $
	\phi_2(z)={(m^2+ns)z+mn+n  \over (ms+s)z+ns+1 }$, then Lemma \ref{lem4} gives that $$
	2\phi_2(0)-4\phi(0)+\phi(0)\phi_2(0)={ 2(mn+n) \over ns+1 }-4n+{ n(mn+n) \over ns+1 }=0.$$
	By a simple calculation, we see that $ 2m-4ns+mn+n=2$.  Since  $m=-1$, we get that $ ns=-1$. Therefore, $m-ns=0$, which means that $\phi$ is not a linear fractional of $\DD$, a contradiction.  The proof is complete.
\end{proof}



\end{document}